\documentclass[10pt]{amsart}
\usepackage[all]{xy}

\usepackage{amsmath,amsthm,amssymb, paralist, xspace, graphicx, url, amscd, euscript, mathrsfs,epic,eepic,verbatim,stmaryrd}
\usepackage{caption}

\usepackage[usenames,dvipsnames]{color}

\usepackage{tikz-cd}

\usepackage[colorlinks=true]{hyperref}

\usepackage[utf8]{inputenc}
\usepackage{chngcntr}
\usepackage{apptools}

\AtAppendix{\counterwithin{alem}{subsection}}
\newtheorem{alem}{Lemma}

\numberwithin{equation}{subsection}

1

\numberwithin{subsection}{section}

\allowdisplaybreaks[1]

\newtheorem*{namedtheorem}{\theoremname}
\newcommand{\theoremname}{testing}

\newtheorem{theorem}[subsubsection]{Theorem}
\newtheorem{proposition}[subsubsection]{Proposition}
\newtheorem{proposition-definition}[subsubsection]
{Proposition-Definition}
\newtheorem{corollary}[subsubsection]{Corollary}
\newtheorem{lemma}[subsubsection]{Lemma}

\theoremstyle{definition}
\newtheorem{adef}[alem]{Definition}

\newtheorem{definition}[subsubsection]{Definition}

\newtheorem{remark}[subsubsection]{Remark}

\newtheorem{examples}[subsubsection]{Example}

\renewcommand{\thesubsubsection}{\ifnum\value{subsection}=0
	\arabic{section}.\arabic{subsubsection}%
\else
	\arabic{section}.\arabic{subsection}.\arabic{subsubsection}%
\fi}

\makeatletter

\@addtoreset{subsubsection}{section}

\let\c@equation\c@subsubsection

\let\subsection@old\subsection
\def\subsection#1{\ifnum\value{subsubsection}>0 \ifnum\value{subsection}=0
	\setcounter{subsection}{\value{subsubsection}}%
\fi \fi
\subsection@old{#1}}
\makeatother

\theoremstyle{remark}

\newcommand{\cX}{{\mathcal X}} 
\newcommand{\cY}{{\mathcal Y}} 

\newcommand\cS{\mathcal{S}}

\newcommand\bA{\mathbf{A}}

\newcommand{\double}{\genfrac..{0pt}1
{\raise -1pt\hbox{$\scriptstyle\longrightarrow$}}{\raise 3pt\hbox
{$\scriptstyle\longrightarrow$}}}

\begin{document}
\large
\title[Stacky factorization]{\large Factorization for stacks and boundary complexes \\ }
\author{ Alicia Harper}
\date{\today}

\maketitle

\begin{abstract}
We prove a weak factorization result on birational maps of Deligne-Mumford stacks, and deduce the following: Let $U \subset X$ be an open embedding of smooth Deligne-Mumford stacks such that $D = X-U$ is a normal crossings divisor, then the the simple homotopy type of the boundary complex $\Delta(X,D)$ depends only on $U$.
\end{abstract}

\section {Introduction}

A recent result of Bergh \cite[Corollary 1.4]{Bergh} allows one to reconstruct a Deligne-Mumford stack from an algebraic space by a sequence of well understood morphisms such as root stacks, gerbes, and blow-ups of smooth centers. On the other hand, the weak factorization theorem of  \cite[Theorem 0.1.1]{AKMW}\cite{W-Cobordism} provides a tool for systematically decomposing a birational map into a sequence of relatively simple birational morphisms. More recently a variant of the weak factorization theorem, applicable more broadly to stacks and other geometric categories, was developed in \cite[Theorem 1.4.1]{AT2}.   Our main theorem can be regarded as a partial amalgamation of these two results, yielding a factorization theorem that takes into account both the stack structure as well as the underlying birational geometry.

\begin{theorem}
\label{main1}
Let $f: X_1 \dashrightarrow X_2$  be a birational map of smooth Deligne-Mumford stacks isomorphic over an open subset $U$, with the complements being normal crossings divisors $D_1 = X_1 - U$ and $D_2 = X_2 - U$, then there exists a sequence of rational maps  

$$ X_1 = Y_0 \dashrightarrow Y_1 \dashrightarrow \cdots Y_n = X_2 $$

such that

\begin{enumerate}
\item Each $Y_i$ is a smooth Deligne-Mumford stack and either $\phi_i : Y_i \to Y_{i+1}$ or $\phi_i :Y_{i+1} \to Y_i$ is a morphism isomorphic over $U$. 
\item The morphism $\phi_i$ is either a blow up of a smooth center having normal crossings with the divisor $Y_i - U$ (resp. $Y_{i+1}$ in the case of a morphism  $Y_i \to Y_{i+1}$) or a root stack construction centered on a component of the divisor $Y_i - U$ (resp. $Y_{i+1} - U$). 
\end{enumerate}
\end{theorem} 

One area where this theorem yields new insight is the study of boundary complexes. Answering a question in \cite[Section 5.3]{CGP}, we prove:

\begin{theorem}
\label{main2}
If $X_1$ and $X_2$ are smooth proper Deligne-Mumford stacks isomorphic over an open set $U$ and $D_1 = X_1-U$ and $D_2=X_2-U$ are both normal crossings divisors, then the associated boundary complexes $\Delta(X_1,D_1)$ and $\Delta(X_2,D_2)$ have the same simple homotopy type.
\end{theorem}

Recall that the boundary complex of a \textit{simple} normal crossing divisor on a smooth projective \textit{variety} $X$ is a combinatorial invariant encoding the way the irreducible components of $D$ intersect. It is a $\Delta$-complex with vertices corresponding to irreducible components of $D$ and a $k$-cell for every non-empty component of a  $k$-fold intersection. In the case where $X_1$ and $X_2$ are both complete varieties, it is a theorem of Stepanov \cite{stepanov} that the simple, or piece-wise linear, homotopy type of the boundary complex $\Delta(X_1,D_1)$ only depends on the open set $U$.  A generalized formulation of this result is given in \cite[Theorem 1.1]{payne1}. These results were in turn foreshadowed by ealier results:  the homological type was independent of the specific compactifiction - a consequence of work done in mixed Hodge theory.  We also note that one of the earliest results on independence of the homotopy type of boundary complexes is due to Danilov \cite{danilov}, and he went as far as possible without access to the weak factorization theorem. For a discussion of the history, we refer to Payne \cite[Section 1]{payne1}. In this paper, we will study the boundary complexes associated to  pairs $(Y,E$) where $Y$ is a smooth Deligne-Mumford stack and $E$ is a normal crossings divisor, this necessitates a more involved definition that was first stated in \cite[Section 5.2]{CGP} and recalled in Definition \ref{bcdef2} below.

Boundary complexes, despite being a seemingly coarse invariant of an algebraic variety, have come to play an increasingly prominent role in work around mirror symmetry \cite{gs1} \cite{ks1}. Indeed, given a toric degeneration of a Calabi-Yau variety $X$, Gross and Siebert have conjectured that a boundary complex enriched with a particular affine structure and associated with the special fiber, is precisely the base manifold relating $X$ with its mirror $\tilde{X}$. 

The proof of Theorem \ref{main2} given below is similar to \cite{stepanov}, though it necessarily relies on Theorem \ref{main1}. To prove Theorem \ref{main1}, we will utilize a result of Abramovich and Temkin\cite[Theorem 1.4.1]{AT2} enabling weak factorization in the category of stacks. However, before we can employ the weak factorization theorem, we must first relate $X_1$ and $X_2$ by a \textit{representable} morphism. Establishing this relationship is the main technical obstacle of the proof and it depends on a recent destackification result of Bergh  \cite[Corollary 1.4]{Bergh}

\subsection{Acknowledgements} I would like to thank Dori Bejleri for several useful comments and suggestions,  Melody Chan for explaining various combinatorial constructions, generously answering many questions, and providing several instructive examples. Above all, I would like to thank Dan Abramovich for outlining this project,  as well as his feedback and support at every step of it.

\subsection{Assumptions}

We will work exclusively over a field of characteristic 0. We shall write $(X,D)$ for a pair consisting of a smooth Deligne-Mumford stack and a simple normal crossings divisor $D$.
\section {Background}

Below, we will review the notion of a boundary complex, the functorial destackification theorem of  \cite{Bergh}, and the weak factorization theorems of \cite{AKMW} and \cite{AT2}.

\subsection{Boundary Complexes}

To motivate the construction, we begin with the simplest case.

\begin{definition}
Let $(X,D)$ be a pair consisting of a smooth Deligne-Mumford stack and a simple normal crossings divisor (an \textit{snc} divisor) and denote by  $D_1 , \ldots , D_r$  the irreducible components of $D$. The \textit{classical boundary complex} is a $\Delta$-complex and it is constructed as follows: It has a single vertex for each $D_i$, and a $k$-simplex for each irreducible component appearing in the intersection of $k$ distinct divisors $D_{i_1},..,D_{i_k}$. The $k$-simplex is glued to the lower dimensional simplicies in the obvious manner. We shall denote this object  by $\Delta^{cl}(X,D)$.
\end{definition}

To illustrate the basic structure of such objects, we recall an example of \cite[Example 2.3]{payne1}.

\begin{examples}
 In $\mathbb{P}^2$, we may find three lines $L_1,L_2,L_3$ in general position and take $D$ to be their union. The corresponding $\Delta$-complex is simply a triangle with vertices corresponding to the lines, and edges corresponding to the intersections. Note that the general position hypothesis is required to ensure that the divisor has simple normal crossings.

The theorem of \cite{stepanov} recalled below can be easily demonstrated here: If one blows up a single intersection point $P$, we obtain a new variety $X$ and the total transform of the divsor $D$ now has four irreducible components and the corresponding boundary complex is a square, a space homotopically equivalent to the triangle.
\end{examples}

Before stating the next theorem, we recall that two CW-complexes or simplicial sets are said to be \textit{simple homotopy equivalent} if they may be related by a sequence of collapsing or expanding $n$-cells. A definition suited to our purposes may be found in \cite[Remark 5.1]{payne1}.

\begin{theorem}
\cite[Theorem 1.1]{payne1} Let $(X,D)$ be a pair consisting of a variety $X$ and a simple normal crossings divisor $D$, then the simple homotopy type of the boundary complex $\Delta^{cl}(X,D)$ depends only on $U = X - D$. 
\end{theorem}

Whatever combinatorial realization of the boundary complex we use, the structure of an algebraic stack, regarded here as a coequalizer of certain algebraic spaces in a presentation, translates into the requirement that this class of combinatorial objects is closed under the operation of taking colimits.  For the purposes of this paper, we will make use of the formalism of \textit{generalized $\Delta$-complexes} to define the boundary complex. Elements of these constructions are recalled in the Appendix below, we refer to \cite{ACP} and \cite{CGP} for a more comprehensive treatment. We shall denote generalized $\Delta$-complexes by the notation $\Delta^{gen}(X,D)$ and unordered $\Delta$-complexes by $\Delta^{un}(X,D)$. Unordered $\Delta$-complexes are sufficent to deal with simple normal crossing divisors, but the extra generality of generalized $\Delta$-complexes is required for normal crossing divisors.

\begin{definition}{\sc Boundary complex of a stack: }
\label{bcdef2}
Given a pair consisting of a Deligne-Mumford stack $X$ and a normal crossings divisor $D$ on $X$, we can take an \'etale surjection $Z \to X$ such that $Z_{D} = D \times _{X} Z$ is an snc divisor. Let $D_{Z \times _{X} Z}$ be the pullback of $D$ to $Z \times _{X} Z$. Then we have two maps $$\Delta^{un}(Z \times _{X} Z,D_{Z \times _{X} Z}) \double \Delta^{un}(D_{Z})$$ corresponding to the  projection morphisms. The \textit{boundary complex of the stack $X$ and divisor $D$} is a generalized $\Delta$-complex corresponding to the coequalizer of the two morphisms. We shall write $\Delta(X,D)$ for $\Delta^{gen}(X,D)$.
\end{definition}

It must be checked that the above definition is independent of the choice of presentation, a proof may be found in \cite[Section 5.2]{CGP}.

\subsection{Destackifcation}

Finally, we will make extensive use of a destackification result of Bergh which we recall, along with some preliminary definitions, below.

\begin{definition}{\sc Destackification: }
 \cite{Bergh} Let $X$ be a smooth algebraic stack over a field $k$, then a destackification is a proper birational morphism $f: X' \to X$ such that the coarse moduli space $(X')_{cs}$ is smooth
\end{definition}

\begin{definition}{\sc Smooth stacky blow-ups: }
\label{ssbu}
Following  \cite{Bergh}, we recall that if $(X,E)$ is a pair consisting of a smooth algebraic stack and $E$ an effective cartier divisor on $X$ with simple normal crossings, then a \textit{ smooth stacky blow-up} is either a root construction along a component of $E$ or the blow-up of a closed smooth substack \textit{intersecting $E$ with normal crossings}.
\end{definition}

In \cite{Bergh}, Bergh provides a functorial construction of destackification morphisms $f: X' \to X$. We state only a special case of his result:

\begin{theorem}
 \cite[Corollary 1.4]{Bergh} Let $X$ be a smooth Deligne-Mumford stack of finite type.  If $X$ has finite inertia, then there exists a smooth, stacky blow-up sequence 
$$(X_m, E_m) \to \cdots \to (X_0,E_0)=(X,E_0)$$
where $X_m \to X $ is a destackification. Moreover, the coarse map $X_m \to (X_m)_{cs}$ can be factored as a gerbe followed by a sequence of root constructions. The construction is functorial with respect to smooth, stabilizer preserving morphisms $X' \to X$.
\end{theorem}

\begin{remark}
Bergh's theorem cannot be applied to our situation naively as the open set $U$ in Theorem \ref{main1} may carry the structure of a stack. Our work below shows that we may extract and isolate just the operations that leave $U$ invariant.
\end{remark}

The role of the divisors $E$ in the statement above is a bookkeeping role analogous to keeping track of exceptional divisors in Hironaka's algorithm, and so by abuse of notation, we implicity regard $E_0$ as the `empty divisor' in the above formulation.

\subsection{Weak Factorization}

In their paper \cite{AT2}, Abramovich and Temkin prove a generalization of the original weak factorization theorem \cite{AKMW}. We refer to Abramovich-Temkin \cite{AT2} for the precise definition of a \textit{weak factorization} and the full statements of their results. Here, we simply recall one of their results in the precise form that we shall use

\begin{theorem}
 \cite[Theorem 1.4.1]{AT2} Assume $f: (X_1,D_1) \to (X_2,D_2)$ is a blow-up of Deligne-Mumford stacks in characteristic 0  such that $X_1 - D_1$ is mapped isomorphically onto $X_2 - D_2$. Then the morphism $f$ admits a weak factorization $$(X_1,D_1) = (V_0,D_0) \dashrightarrow V_1 \dashrightarrow \cdots \dashrightarrow (V_n,D_n) = (X_2,D_2) $$ where the associated morphisms are representable, trivial over $U = X_2 - D_2$, and are constructed by blowing up of smooth centers intersecting the $D_i$ (respectively $D_{i+1}$) with normal crossings.
\end{theorem}

\section{Groupoids in Algebraic Stacks}

In this section, we provide a definition for a groupoid in stacks together with the rudiments of $2$-topoi particular to our situation. We begin by recalling the case of groupoids in schemes

\begin{definition}
\label{gpddef1}
A \textit{classical} groupoid presentation in schemes consists of a septuple $(U,R, \pi_1: R \to U,\pi_2: R \to U,c: R \times_{\pi_1,U,\pi_2} R \to R,e:U\to R,i: R \to R)$ where the morphisms are subject a standard list of axioms for which we refer to \cite[Tag:  0230]{stacks}.
\end{definition}

The axioms contained in the reference are a simple categorification of the ordinary definition of a groupoid. Such an object does not immediately correspond to an algebraic stack, but it will admit an associated quotient under the assumption that $\pi_1$ and $\pi_2$ are smooth. 

To better understand the need for this categorical detour, we recall that a standard method of studying a Deligne-Mumford stack $\cX$ consists of replacing $\cX$ with a groupoid in schemes $(U,R,\pi_1,\pi_2, c,e,i)$ such that $[U/R]$ is isomorphic to $\cX$. Below, we shall sometimes abuse notation and write  $(U \double R,c)$  for our groupoids $(U,R,\pi_1,\pi_2, c,e,i)$ .  The gist of the idea is that statements proven on $U$ and $R$ in an equivariant way can then be descended to statements concerning $\cX$. This is fine, but a complication arises if instead we study a morphism $f: \cY \to \cX$: A groupoid presentation $(U,R,\pi_1,\pi_2, c,e,i)$ of $\cX$ will determine a groupoid presentation of $\cY$, but it will \textit{not} be a groupoid presentation \textit{in schemes} unless $f$ was already representable. 

It is quite possible that one could start with a groupoid presentation in schemes, apply some geometric procedure, and end up with a groupoid presentation in stacks. For example, it may be the case that $U$ and $R$ both carry the action of some group $G$ and the structure morphisms are equivariant with respect to this action, then one may find themselves tempted to replace $U$ with $R$ with their stack quotients $[U/G]$ and $[R/G]$.  It is therefore desirable to sketch the foundations of groupid presentations in Deligne-Mumford stacks and we believe that this most naturally done in the language of $\infty$-categories.

These generalized groupoids of Deligne-Mumford stacks are defined in Definition \ref{repgrpdef}. In Proposition \ref{repquotprop}, we verify the existence of algebraic quotients in a special case. In Proposition \ref{stablepullback}, we verify the crucial fact that the pullback of an inertia stable (Definition \ref{iertstab}) groupoid is still inertia stable. The latter fact is used to descend the operations of Bergh's destackification in the next section.

In the case of schemes, it is known that $[U/R]$ may be regarded as the $2$-coequalizer of $U \double R$ \cite[Tag:  044U]{stacks}. On the other hand, it is also known that $R$ may be regarded as the kernel pair associated to the morphism $U \to [U/R]$  \cite[Tag: 04M9]{stacks}. This entire situation may be expressed succinctly: The morphism  $U \to [U/R]$ is an \textit{effective epimorphism}. 

\begin{definition}
In an ordinary category $C$ with all small limits, an effective epimorphism $f:c \to d$ is a morphism such that $(c \times_{d} c) \double c \to d$ is a coequalizer diagram.
\end{definition}

In topos theory, it is known that every congruence has an effective quotient. This notion and result admits a generalization to $(n,1)$-topoi for $1 \leq n \leq \infty$. We assume some familiarity with the theory of $\infty$-categories, but we recall the definitions most relevant to us below. 

\begin{definition}
A morphism $f: X \to Y$ of topological spaces is $n$-truncated if the homotopy groups $\pi_i (F,x)$ of the homotopy fiber $F$  of $f$ vanish for all $i > n$ and all base points $x$.
\end{definition}

\begin{definition}
\cite[Definition 6.1.2.2]{HTT}  Let $\cX$ be an  $\infty$-category. A simplicial object of $\cX$ is a map of $\infty$-categories  $:U_{\bullet}: N(\Delta)^{op} \to \cX$. 
\end{definition}

Recall that $\Delta$ is the category with objects $[n]$ and morphisms $[n] \to [m]$ corresponding to all set functions $\{0,\ldots,n\}\to \{0,\ldots,m\}$. Fixing a simplicial object $U_{\bullet}$ and following Lurie \cite[Notation 6.1.2.5.]{HTT}: If $K$ is a simplicial set, then we take $\Delta_{\slash K}$ to be the category of pairs $(J,\eta)$ where $J \in \Delta$ and $\eta \in Hom_{SSet}(\Delta^J,K)$. The definition of $\Delta^J$ may be found in \cite[Definition 1.1.1.5]{HTT} and it is only a slight generalization of the simplicial set associated with the simplex $\Delta^n$.

\begin{definition}
\label{infinitygroupoids}
\cite[Definition 6.1.2.6]{HTT} Let $\cX$ be an  $\infty$-category. A simplicial object $U_{\bullet}$  in $\cX$ is said to be a groupoid object if, for every $n \geq 2$ and every $0 \leq i \leq n$, the induced map $\cX_{\slash U[\Delta^n]} \to \cX_{\slash U[\Lambda_i^n]}$ is a weak equivalence

\end{definition}

To motivate these notions, we relate the notion of an $\infty$-groupoid with the more classical concept.

\begin{proposition}
\cite[Remarks 2.3]{PRID}\cite{glenn}There is a one to one correspondence between groupoid objects in $N(Sch(R))$and groupoid presentations defined over $N(Sch(R))$.
\end{proposition}

\begin{proof}
Since $Sch(R)$ is a $1$-category, all categorical constructions are the classical ones, the term weak equivalence in Definition \ref{infinitygroupoids} simply means isomorphism. Now, we essentially just follow the references cited: Clearly we have natural projection maps $\partial_i: U_1 \to U_0)$ where $i=0$ or $i=1$. The natural map $U_2 \to (U_1 \times_{\partial_0,U_0,\partial_1} U_1)$ is an isomorphism, and we denote by $c$ the composition 
 $$U_1 \times_{\partial_0,U_0,\partial_1} U_1 \to U_2 \to U_1$$
where the left map is just the inverse of $(\partial_2, \partial_0)$ and the right map is $\partial_1$. We define an inverse map $i: U_1 \to U_1$ according to the composition
$$U_1 \to (U_0 \times U_1) \to (U_1 \times_{\partial_1,U_0,\partial_1} U_1) \to U_2 \to U_1$$
where the first morphism is $(\partial_0,id)$, the second morphism is $(s_0,id)$, the third $(\partial_1,\partial_2)^{-1}$, and the fourth is $\partial_0$. A third morphism $e$ corresponding to the identity may be constructed either directly or from $i$ and $c$.

It can be shown that the data $(U_0, U_1, \partial_0, \partial_1, c,)$ satisfies the axioms of a groupoid in schemes \cite{PRID}[Lemma 2.12 and Remarks 2.13]. On the other hand, the definition of a groupoid object ensures that there is an isomorphism $(U_1 \times_{\partial_0,U_0\partial_1} \ldots \times_{\partial_0,U_0\partial_1} U_1) \to U_n$ where the fiber product involves $n$ copies of $U_1$, thus the entire $\infty$-groupoid object may be recovered from the data of $(U_0, U_1, \partial_0, \partial_1, c)$.
\end{proof}

\begin{definition}
\cite[Definition 6.4.3.1]{HTT} Let $\cX$ be an  $\infty$-category and $U_{\bullet}$ a groupoid object of $\cX$, then $U_{\bullet}$ is \textit{n-efficient} if the induced map of morphism spaces $Map_{\cX}(E,U_1) \to Map_{\cX}(E,U_0 \times U_0)$ is $(n-2)$ truncated for every object $E$.
\end{definition}

Stacks are also known in the literature as $(2,1)$-sheaves on the \'etale \cite{Lurie2}[1.2.5] and it is well known that such sheaves may be organized into a $2$-category carrying the structure of a $(2,1)$-topos. From the theory of $(2,1)$-topoi, we need only the following result \cite[Theorem 6.4.1.5]{HTT}:

\begin{proposition}
\label{effgrp}
 \cite[Theorem 6.4.1.5]{HTT} In any $(n,1)$-topos, every n-efficient groupoid object $U_{\bullet}$ is effective and thus admits a colimit $U_{-1} $ such that the natural map $U_1 \to U_0 \times_{U_{-1} } U_0$ is an equivalence.
\end{proposition}

\begin{remark}

As is described in \cite[1.2.5 - Deligne-Mumford Stacks as Functors]{Lurie2}, the $2$-category of Deligne-Mumford stacks over a ring $R$ embeds fully and faithfully into the $\infty$-category of functors $Sch(R) \to \tau_{\leq 1}(Top)$ where $\tau_{\leq 1}(Top)$ is the $\infty$-category of spaces with vanishing homotopy groups $\pi_{i}(X)$ for $i>2$. After applying an appropriate sheaf condition and a local trivality condition, we recover the $(2,1)$-category of Deligne-Mumford stacks. In more classical language, we see this construction is essentially the same as regarding a Deligne-Mumford stack as a  psuedofunctor. 

One also observes that the $2$-efficient condition on a groupoid $U_{\bullet}$ can be understood in this situation as follows: Fix a morphism $Spec(S) \to Spec(R)$ and let $h_S$ be the associated functor of points, then $U_{\bullet}$ is $2$-efficent exactly when the morphism $Map_{\cX}(h_S,U_1) \to Map_{\cX}(h_S,U_0 \times U_0)$ is 0-truncated. for every such $Spec(S)$. Chasing definitions, this is equivalent to the morphism of groupoids $U_1(S) \to U_0(S) \times U_0(S)$ having a homotopically trivial fiber. This will be true only when $Iso_{U_1/U_0 \times U_0}(x,x)$ is trivial for each $x \in U_1(S)$ and follows from this that the morphism $U_1 \to U_0 \times U_0$ must be representable. Working backwards, we see that representability of $U_1 \to U_0 \times U_0$ is a sufficent condition or a morphism to be $2$-efficent.
\end{remark}

With these considerations in mind, we are led to the following definition \cite{km1} \cite[Tag: 043T]{stacks}:

\begin{definition}
\label{repgrpdef}
A \textit{representable groupoid in Deligne-Mumford stacks} over a scheme $S$  is groupoid object $U_{\bullet}$ in $N(DM(S))$ with representable projection morphisms where $DM(S)$ is  the $(2,1)$-category of Deligne-Mumford stacks.
\end{definition}

Note that Definition \ref{effgrp} implies that effective groupoids in algebraic stacks have associated quotient stacks $[U/R]$. The next proposition verifies the existence of a quotient stacks for representable groupoid in Deligne-Mumford stacks, and it also allows us to replace a representable groupoid in Deligne-Mumford stacks with a groupoid in schemes as defined in Definition \ref{gpddef1}.

\begin{proposition}
\label{repquotprop}
Assume $U_{\bullet}$ is a representable groupoid in Deligne-Mumford stacks over $S$, then  $U_{\bullet}$ is $2$-efficient and thus effective. If the projections $\pi_i: U_1 \to U_0$ are smooth, then the associated quotient stack is algebraic and we may regard it as the quotient of a classical groupoid in schemes $(V,R_V,\pi'_1,\pi'_2, c',e',i')$.
\end{proposition}

\begin{proof} 
Choose an \'etale surjection from a scheme $V$ to $U_0$. This in turn induces an \'etale surjection $f: (V \times V) \to (U_0 \times U_0)$. In the $(2,1)$-topos $DM(S)$, the map $f$ is an effective epimorphism \cite[Tag:  044U]{stacks}. Denote by $g: R_{V\times V} \to V\times V$ the pullback of $U_1 \to U_0\times U_0$, then we obtain a groupoid in schemes $(V,R_V,\pi'_1,\pi'_2, c',e',i')$. The fact that the morphisms $\pi_i$ are representable imply that $g$ is a representable morphism of schemes. Thus the map $g$ is $0$-truncated and it follows from \cite[Proposition 6.2.3.17]{HTT} that the original map $U_1 \to U_0\times U_0$  is also $0$-truncated and the groupoid $(V,R_V,\pi'_1,\pi'_2, c',e',i')$ is effective.

Consequently, we have a quotient stack $[U/R]$. By gluing the four small cartesian squares in the diagram below in pairs, and then gluing the pairs, we see that that the big cartesian square is cartesian. This, together with the fact that the composition the effective epimorphisms $V \to U_0$ with the effective epimorphism $U_0 \to [U/R]$  is an effective epimorphism \cite[Corollary 7.2.1.15]{HTT}, shows that $[U/R]$ is in fact the quotient of  $(V,R_V,\pi'_1,\pi'_2, c',e',i')$ and is thus an algebraic stack. 
\begin{equation}
\label{pullback}
\begin{tikzcd}[row sep=scriptsize, column sep=scriptsize]
R_V  \arrow[r] \arrow[d]&V \times_{\pi_{0}} R  \arrow[r] \arrow[d]& V \arrow[d]\\
V \times_{\pi_{1}} R  \arrow[r] \arrow[d]&R  \arrow[r, "\pi_{0}"] \arrow[d, "\pi_{1}"]& U \arrow[d, "t"]\\
V  \arrow[r]&U  \arrow[r, "s"]&  (U/R)
\end{tikzcd}
\end{equation}
\end{proof}

From the argument above, we also obtain the corollary:

\begin{corollary}
\label{schpres}
Suppose we have a representable \'etale surjection of smooth Deligne-Mumford stacks $U \to X$. From this, we obtain a classical presentation $(U,R,\pi_1,\pi_2, c,e,i)$ of $X$ as a groupoid in algebraic stacks. If we also have a smooth scheme $V$ and an \'etale surjection morphism $V \to U$ giving rise to a second presentation $(V,R_V,\pi'_1,\pi'_2, c',e',i')$ of $X$ in schemes over $S$, then there natural morphism of groupoid presentations $(V,R_V,\pi'_1,\pi'_2, c',e',i') \to (U,R,\pi_1,\pi_2, c,e,i)$.
\end{corollary}

For the purposes of this paper, we will place special emphasis on a particular class of presentations.

\begin{definition}
\label{iertstab}
Let $U_{\bullet}$ be a groupoid in Deligne-Mumford stacks over $S$. If the projection morphisms $\partial_i: U_k \to U_{k-1}$ are all inertia preserving morphisms, then we shall say the groupoid is \textit{inertia stable}. 
\end{definition}

\begin{examples}
\label{stableexample}
The groupoid presentation of Deligne-Mumford stack $X$ induced by a representable \'etale surjection $U_{-1} \to X$ from a scheme $U_{-1} $ is always inertia stable since each $U_i$ is a scheme.
\end{examples}

We also note that inertia stable groupoids have representable projections and thus have quotient stacks by Proposition \ref{repquotprop}. Our next goal is to show that groupoid presentations are pullback functorial with respect to a morphisms of quotients.

\begin{proposition}  Let $U_{\bullet}$ be a smooth representable groupoid in Deligne-Mumford stacks over $S$ with algebraic quotient $X$. If $Y$ is a Deligne-Mumford stack and $f: Y\to X$ is a morphism of Deligne-Mumford stacks, then $U_{\bullet}\times_{X} Y$  is a groupoid object with quotient $Y$. If $U_{\bullet}$  is inertia stable, then $U_{\bullet}\times_{X} Y$ is also inertia stable. 
\label{stablepullback}
\end{proposition}

\begin{proof}
The first claim follows from the fact that $X$ is a colimit of  $U_{\bullet}$  and colimits are universal in a $(2,1)$-topos \cite[Theorem 6.4.1.5]{HTT}. The second claim follows from \cite[Tag: 0DUB]{stacks}.  
\end{proof}

\begin{lemma}
\label{closedsub}
Let $U \to X$ be a representable \'etale surjection of smooth Deligne-Mumford stacks. If $Z$ is a closed substack of $U$, $ \pi_i: U\times_{X} U \to U$ are the projections, and $\pi^{-1}_0 (Z) = \pi^{-1}_1 (Z)$, then $Z$ descends to a closed substack $Z_X$ of $X$. 
\end{lemma}
\begin{proof}

Replacing $V$ with $Z$ in Figure \ref{pullback}, we see that we have a closed subgroupoid $(Z,R_Z,\pi'_1,\pi'_2, c',e',i')$ of $(U,R,\pi_1,\pi_2, c,e,i)$. The projections of the subgroupoid are \'etale, surjective, and representable and it follows that $Z$ descends to a closed substack $Z_X$ of $X$.
\end{proof}

\subsection{Relative coarse moduli spaces}

We now recall the Keel-Mori theorem \cite{km1}, in particular its phrasing in the language of stacks \cite{conrad2}.

\begin{definition}
\cite{conrad2} The \textit{coarse moduli space} of an Artin stack $\cX$ over a scheme $S$ is a map $f: \cX \to X$ satisfying the following properties:
\begin{enumerate}
\item The morphism $f$ is initial among maps to algebraic spaces over $S$
\item For every algebraically closed field $k$, there is a bijection between elements of $X(k)$ and isomorphism classes of $\cX(k)$.
\end{enumerate}
\end{definition}

\begin{theorem}
(Keel-Mori) \cite{km1} \cite[Theorem 1.1]{conrad2}
\label{keelmori}
Given a scheme $S$ and a separated Artin stack $\cX$ of locally finite presentation over $S$, there exists a coarse moduli space $X$ of finite presentation over $S$ with finite inertia stack $I_S (\cX)$. The space $X$ satisfies the following properties.

\begin{enumerate}

\item The map $\pi: \cX \to X$ is proper and quasi-finite

\item If $X' \to X$ is a flat map of algberaic spaces, then $\pi': \cX \times_{X} X' \to X'$ is also a coarse moudli space.
\end{enumerate}

\end{theorem}

Let us also record a lemma of Abramovich and Vistoli.

\begin{lemma}
\label{avref}
\cite[Lemma 2.2.2]{av1} Let $\cX \to X$ be a proper quasifinite morphism from a Deligne-Mumford stack to a noetherian scheme $X$ and let $X' \to X$ be a flat morphism of schemes, and denote $\cX' = X'\times_{X}\cX$.
\begin{enumerate}
\item If $X$ is the moduli space of $\cX$, then $X'$ is the moduli space of $\cX'$.
\item If $X' \to X$ is also surjective and $X'$ is the moduli space of $\cX'$, then $X$ is the moduli space of $\cX'$.
\end{enumerate}
\end{lemma}

It is desirable for our principal application to be able to form relative coarse moduli spaces $\cX_{r. cs}$ where the base is a Deligne-Mumford stack $\cS$ and the morphism $\cX_{r. cs} \to \cS$ is taken to be representable. We formulate this notion in the following definition.

\begin{definition}
The \textit{relative coarse moduli space} of an Artin stack $\cX$ over an Artin stack $\cS$ is a morphism: $\cX \to \cX_{r. cs}$ over $\cS$ such that for any smooth surjection $U \to \cS$ from a scheme $U$, the pullback morphism $\pi_{U}: \cX_U \to \cX_{r. cs,U} $ is a coarse moduli space in the ordinary sense.
\end{definition}

The following theorem enables the construction of such relative coarse moduli spaces in our setting. We recall certain aspects of it's proof in our statement of the theorem, as we will utilize them below.

\begin{theorem}
\label{rcmthm} \cite[Theorem 3.1]{AOV3}
Given a morphism of irreducible Deligne-Mumford stacks $\cX \to \cS$ and an \'etale surjection $U \to \cS$ from a scheme $U$, denote $U_X := U \times_{\cS} \cX$. Then there is a diagram of the following form.

\begin{equation}
\label{rcm}
\begin{tikzcd}[row sep=scriptsize, column sep=scriptsize]
U_X \times_{\cX} U_X \arrow[dr] \arrow[dd,bend left = 30] \arrow[rr, shift right = 2, "\pi_{2,1}"]   \arrow[rr, shift left=2, "\pi_{1,1}"]& & U_X  \arrow[dr]  \arrow[rr] & & \cX \arrow[dr] \arrow[dd,bend left = 30, crossing over] \\
& (U_X \times_{\cX} U_X)_{cs} \arrow[dl] \arrow[rr, shift left = 2, "\pi_{1,0}"]  \arrow[rr, shift right=2, "\pi_{2,0}"]& & (U_X)_{cs}  \arrow[dl]  \arrow[rr]  & & \cX_{r. cs} \arrow[dl] \\
U \times_{\cS} U  \arrow[rr, shift right = 2, "\pi_{2,2}"] \arrow[rr, shift left=2, "\pi_{1,2}"]& & U \arrow[from = uu, crossing over, bend left = 30]  \arrow[rr]   & & \cS \arrow[from = uu,bend left = 30, crossing over]
\end{tikzcd}
\end{equation}
The diagram is commutative in the natural way. It satsifies the following additional properties.
\begin{enumerate}
\item For $i \in \{1,2\}$, the morphisms $\pi_{i,0}$ are \'etale and surjective, and the following diagram is 2-cartesian
\begin{equation}
\label{rcm2}
\begin{tikzcd}
U_X \times_{\cX} U_X  \arrow[r, "\pi_{i,1}"] & U \\
(U_X \times_{\cX} U_X)_{cs} \arrow[from = u] \arrow[r, "\pi_{i,0}"] & (U_X)_{cs} \arrow[from = u]
\end{tikzcd}
\end{equation}

\item The morphism $\cX_{r. cs} \to \cS$ is representable, and the morphism $U_X \to (U_X)_{cs}$ is the pullback of the morphism $ \cX \to \cX_{r, cs} $ by $ U_X \to \cX $. In particular, the morphism $ \cX \to \cX_{r, cs} $ is the desired relative coarse moduli space of of $\cX \to \cS$.

\end{enumerate}

\end{theorem}
\begin{proof}

The morphism $U_X \to \cX$ is induced by pullback of $U \to \cS$ and likewise $U_X \times_{\cX} U_X$ may be regarded as the pullback of $U_X$ under either projection $(U \times_{\cS} U) \to U$. The existence of the factorization $\cX \to \cX_{r. cs} \to \cS$ follows immediately from \cite[Theorem 3.1]{AOV3}.  Note that we may also factor $U_X \to U$ as  $U_X \to (U_X)_{cs} \to U$ and likewise $(U_X \times_{\cX} U_X) \to (U \times_{\cS} U)$ as  $U_X \times_{\cX} U_X \to (U_X \times_{\cX} U_X)_{cs} \to U \times_{\cS} U$. The horozontal morphisms arise from the fact that the pullback of a coarse moduli map $\cY \to Y$ by a representable flat morphism is also a coarse moduli map\label{avref} \cite[Theorem 3.1]{AOV3}.

\end{proof}

\subsection{Descent of destackification operations}
We now show that the destackification operations of Definition \ref{ssbu} utilized in  \cite{Bergh} can be relativized in our situation.

\begin{corollary}
\label{opdescent}
Fix a pair $(X,D)$ where $X$ is a smooth Deligne-Mumford stack over a scheme $Y$  and $D$ is a simple normal crossings divisor. Assume $f: V \to X$  is an \'etale surjection of smooth and tame Deligne-Mumford stacks such that the pullback of $X-D$ is a scheme and the projection morphisms $V \times_{X} V \to V$ are inertia preserving, then Bergh's destackification procedure decomposes as a natural sequence of operations, each of which may be descended from $V$ to an operation on $X$ which will pullback to the original operation.
\end{corollary}

\begin{proof}
We utilize the fact the that the algorithm of Bergh is functorial with respect to smooth, surjective, and inertia preserving morphisms.
\begin{enumerate}
\item \textit{Blowing up smooth substacks of $V$} - If $Z$ is a closed substack of $V$ blown up during the course of the algorithm, then its pullback to $V \times_{X} V$ is independent of the choice of the projection used to form the pullback and Lemma \ref{closedsub} ensures that it descends to a closed substack of $X$. The operation of blowing up is functorial with respect to smooth morphisms, so the operation (and not merely $Z$) descends to $X$.

\item\textit {Root stacks}: The algorithm of Bergh only takes roots of components of the divisor $f^{-1}(D)$. These components and the rooting index can be descended to an $X$ since the formation of root stacks in functorial with respect to \'etale morphisms.

\item\textit {Descent of $V \to V_{cs}$ to $X \to X_{r.cs}$}: In order to apply Theorem \ref{rcmthm}, we note that Bergh's algorithm yields a decomposition of the map $V \to V_{cs}$ into a composition of gerbes, and root stacks along components of the simple normal crossings divisor $f^{-1}(D)$. Since $V$ is generically a scheme, the algorithm will use only root stacks. We are now in a situation where Theorem \ref{rcmthm} can be used to construct $X_{r.cs}$ and the morphism $X \to X_{r.cs}$. Since $V \to V_{cs}$ is the pullback of $X \to X_{r. cs}$, we may descend our root stacks  $V \to V_1 \to \cdots \to V_{cs}$ to a sequence of root stacks $X \to X_1 \to \cdots \to X_{cs}$ 
\end{enumerate}

\end{proof}

\section{Invariance of simple homotopy type under the destackifcation operations.}

We show that the operations used in Corollary \ref{opdescent} do not modify the simple homotopy type of the boundary complex. We denote by $(X,D)$ our standard pair where $X$ is a smooth Deligne-Mumford stack and $D$ is a simple normal crossings divisor on $X$

\begin{remark}
\label{keyremark}
One of the fundamental properties of the functorial weak factorization theorem \cite[Theorem 1.4.1]{AT2}  is that it only ever blows up subvarieities that have normal crossing intersection with the strict transform of the original divisors. This ensures that the following lemma applies to any morphism appearing in the decomposition it yields.

Likewise, the smooth blow-ups appearing in Bergh's destackification algorithm \cite{Bergh} only take place over centers $C$ having simple normal crossings intersection with the fixed divisor, so the lemma below applies there as well.

In our application, all morphisms appearing in either the destackification algorithm, or in the functorial weak factorization algorithm, will be isomorphisms along a fixed open subset $U$. 
\end{remark}
\begin{lemma}

\label{centerdescent}
Let $X_1$ be a smooth proper Deligne-Mumford stack with $D_1$ a simple normal crossings divisor on $X_1$ and let $f: (X_2,D_2) \to (X_1,D_1)$ be the blowup of a smooth closed substack $Z$  having normal crossings with $D_1$, then there is a simple homotopy equivalence between $\Delta (X_1,D_1)$ and $\Delta(X_2,D_2)$
\end{lemma}

\begin{proof}
Choose an \'etale surjection $h: U_1 \to X_1$ from a scheme $U_1$ and set $U_2 = U_1 \times_{X_1} X_2$, then the morphism $f_U: U_2 \to U_1$ is a blowup of a smooth subscheme $$ Z^U_1 = Z \times_{X_1} U_1 \subset U_1$$ that has normal crossings intersections with the divisor $D^U_1 = D_1 \times_{X_1} U_1$, consequently by Stepanov's lemma \cite[Lemma 1]{stepanov}, there is a simple homotopy equivalence 
$$\Delta(U_2, D^U_2) \cong \Delta(U_1 , D^U_1 )$$ where $D^U_2 = U_1 \times_{X_1} D_2$.

Since the projection morphisms of the groupoid $$(U_2, U_2 \times_{X_2} U_2,\pi'_1 , \pi'_2, c')$$ are \'etale, we obtain projection morphisms between boundary complexes with appropriate divisors $D^R_2,D^R_1$ induced by pullback:

\begin{equation}
\begin{tikzcd}[column sep=small]
\Delta(  U_2 \times_{X_2} U_2 ,D^R_2)  \arrow[rr] \arrow[d, xshift=0.7ex] \arrow[d, xshift=-0.7ex] & & \Delta(U_1\times_{X_1} U_1,D^R_1) \arrow[d, xshift=0.7ex] \arrow[d, xshift=-0.7ex]   \\
\Delta(U_2, D^U_2) \arrow[rr] & &  \Delta(U_1,D^U_1) 
\end{tikzcd}
\end{equation}

We shall denote the projection morphisms $U_1\times_{X_1} U_1 \double U_1$  by $\pi_i$, and we shall take $g$ be the morphism $U_1\times_{X_1} U_1 \to X_1$ and set $Z^R_1:= g^{-1}(Z) $.

Now suppose that $\Gamma$ is an irreducible component of the intersection of $k$ distinct irreducible components $(D'_{R_1,i_1}, ..., D'_{R_1, i_m})$  of  $D^R_1$ corresponding to some $(k-1)$-simplex of $\Delta(  U_2 \times_{X_2} U_2 ,D^R_2)$. If $Z^R_1$ has proper non-empty normal crossings intersection with $\Gamma$, then the combinatorial operation on the boundary complex induced by blowing up (see \cite{stepanov}) yields a $k$-simplex lying above the simplex corresponding to $\Gamma$ as depicted in the bottom row of Figure \ref{fig:my-label}. On the other hand, if $\Gamma \subset Z^R_1$, then \cite{stepanov} shows that blowing up corresponds to a star subdivision of the $k-1$-simplex associated $\Gamma$ as depicted in the middle row of Figure \ref{fig:my-label}. 

Note that the operations done to the simplex associated to $\Gamma$ in $R_1$ correspond to the same operations done to the simplex associated to $\pi_i(\Gamma)$ in $U_1$  and the component  $\pi_j (\Gamma)$ of  $\pi_j(D'_{R_1, i_1}), ...,\pi_j( D'_{R_1, i_m})$. consequently, the simple homotopy equivalence between $\Delta(U_2,  D^U_2)$ and $\Delta(U_1 ,  D^U_1)$ may be lifted to a simple homotopy equivalence between the boundary complexes of the relations $$ \Delta(R_2, D^R_2) \cong  \Delta(R_1,D^R_1) $$ in a compatible way. Consequently, it follows that there is a simple homotopy equivalence between the generalized boundary complexes $\Delta (X_1,D_1)$ and $\Delta(X_2,D_2)$.

\end{proof}
\begin{figure}
   \center{\includegraphics[width = \textwidth]{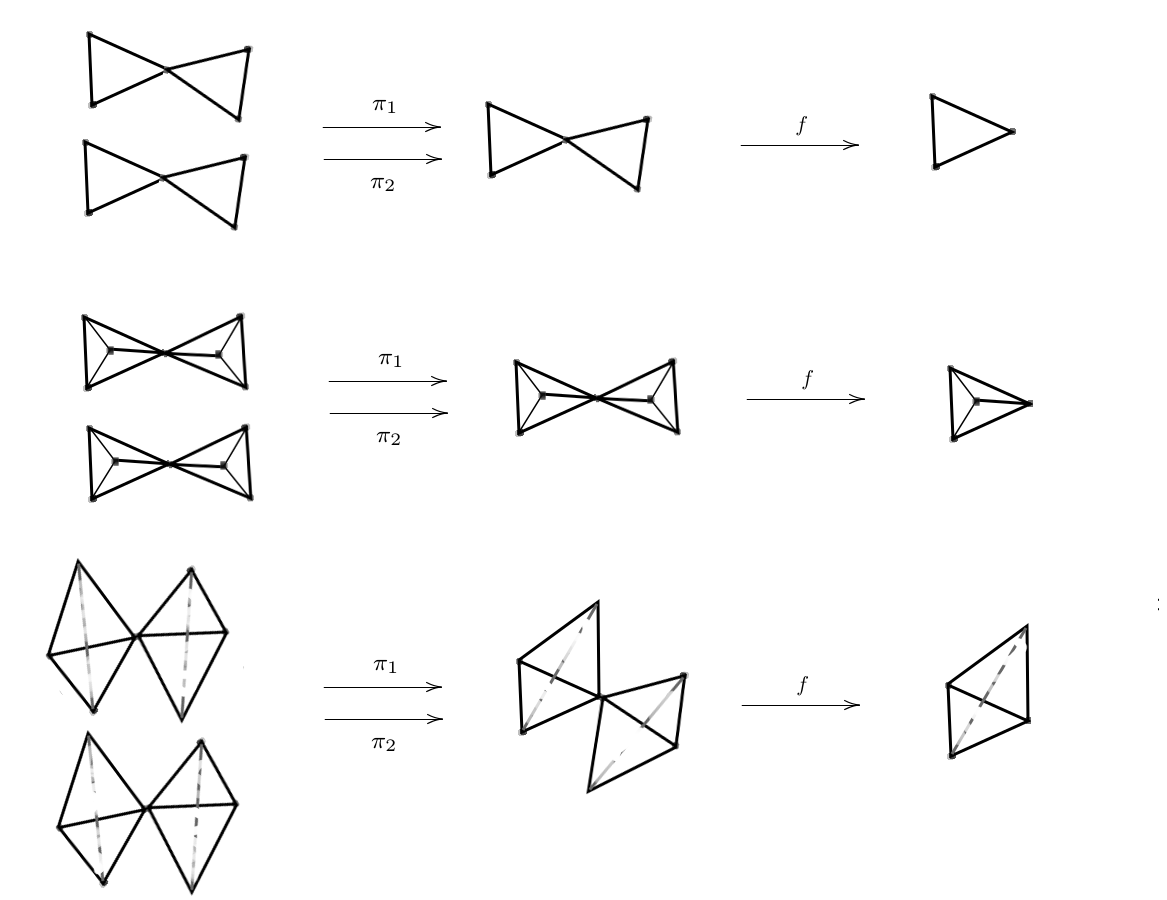}}
\caption{\label{fig:my-label} Two modifications of a presentation of a boundary complex.}
\end{figure}
We show that the formation of root stacks does not alter the boundary complex.

\begin{lemma}
\label{rootdescent}
Let $f: (X,D) \to (X',D')$ be a root construction over a divisorial component $E\subset D$, then the induced morphism of complexes $\Delta f: \Delta (X,D) \to \Delta(X',D')$ is bijective.
\end{lemma}

\begin{proof}
An \'etale  neighborhood of a point $p\in D \subset X$ has an \'etale map to $0\in (\bA^n , x_1 \cdots x_m = 0)$ where $m \leq n$. Let $D_i$ be the component associated to $x_i = 0$ and consider the r-th root stack $(\bA^n)_{r,D_1}$ of the divisor $D_1$. Then \cite[ Example 2.4.1]{cadman} tells us that $(\bA^n)_{r,D_1}$  is isomorphic to 

$$[Spec(A[y]/(y^r - x_1)) / \mu_r ]$$

where $A =k[x_1, \ldots, x_n]$ and $\mu_r$ acts on y by $t \cdot y = t^{-1}y$  and acts trivially on all elements of $A$. Since we have an \'etale surjection  $Spec(k[x_1, \ldots, x_n, y]/(y^r - x_1) \to (\bA^n)_{r,D_1}$, we may form the associated presentation

$$Spec(A[y]/(y^r - x_1)) \times_{ (\bA^n)_{r,D_1}} Spec(A[y]/(y^r - x_1)) \double Spec(A[y]/(y^r - x_1))$$

The relations in the above recognized may be recognized as $$Spec(A[y]/(y^r - x_1])$$ where one projection is the identity and the other sends $y$ to $t^{-1}y$. This is a smooth affine variety.  Note that the pullback of the divisor $\{ x_1 \cdots x_m = 0 \}$ to $(A[y]/(y^r - x_1))$ is $\{ y^r x_2 \cdots x_m = 0 \}$ and this topologically the same as the reduced divisor $\{ y x_2 \cdots x_m = 0 \}$. In particular, no distinct divisors are identified under the projections. From this, it follows that the two induced morphisms of unordered $\Delta$-complexes are equal and thus the boundary complex of $((\bA^n)_{r,D_1},f^{-1}(D))$ is equal to the boundary complex of $(\bA^n , x_1 \cdots x_m = 0)$.

Since the operation of forming a root stack is local in the \'etale topology and our local computation shows that procedure only changes the multiplicity of a divisor, we see that the construction does not introduce any new divisors nor modify the incidence the relations between them.

\end{proof}

\section{The Construction}

To prove the main theorem, we proceed by a sequence of reductions.

\subsection{\sc Step I -  Reduction to the case of SNC divisors}

Given an arbitrary pair $(X,D)$ consisting of a smooth Deligne-Mumford stack and a normal crossings divisor $D$,  we may apply Lemma \ref{redsnc} to find a pair $(X',D')$ consisting of a smooth Deligne-Mumford stack and a simple normal crossings divisor $D'$ together with a proper birational morphism $X' \to X$ that is an isomorphism over $X-D$. As in Lemma \ref{redsnc}, this will induce a simple homotopy equivalence of generalized boundary complexes $\Delta^{gen}(X,D)$ and $\Delta^{gen}(X',D')$.

\subsection{\sc Step II -  Reduction to the case of proper morphisms}

Let $(X_1,D_1)$ and $(X_2,D_2)$ be pairs consisting of a smooth Deligne-Mumford stack and a simple normal crossings divisor. Additionally, we shall suppose that $X_1 - D_1$ is isomorphic to $X_2 - D_2$ and we shall thus regard $X_1$ and $X_2$ as two alternative compactifications of an open substack $U$.

\begin{lemma}
\label{GraphConstruct}
Let $ X_1$, $X_2$, and $U$  be as above, then there exists a smooth and proper algebraic stack $Z$  and a dense open embedding $U \subset Z$ and morphisms $\phi_i:Z \to X_i$  respecting the open set $U$.
\end{lemma}

\begin{proof}
Let $C \subset X_1 \times X_2 $ be the closure of the graph of the rational map  $X_1 \dashrightarrow X_2$. Then it suffices to find a smooth and proper stack Z together with a morphism $Z \to C$  that is an isomorphism over $U$ and such that $Z - U$ is a simple normal crossings divisor. For this, we simply cite the functorial resolution theorem of Temkin \cite[Theorem 5.1.1]{Temkin}.
\end{proof}

\begin{corollary}
\label{MorphismReduction}
If the morphisms $\phi_i$ above induce simple homotopy equivalences of  boundary complexes, then the boundary complexes of $X_1$ and $X_2$ are also simple homotopy equivalent
\end{corollary}

With the two results above in hand, we will now assume the existence of a morphism $(X_1,D_1) \to (X_2,D_2)$ for the remainder of this section. 

\subsection{\sc Step III -  Reduction to the case of proper representable morphisms:}

In the previous section, we obtained a morphism $(X_1,D_1) \to (X_2,D_2)$. Our next goal is to reduce to the case where this morphism is actually representable using Corollary \ref{opdescent}. 

To do this, choose an \'etale morphism $Y_2 \to X_2$ where $Y_2$ is a scheme and let $Y_1 = X_1 \times _{X_2} Y_2$. $Y_2$ will not generally be a scheme, but the failure of the morphism $Y_1 \to Y_2$ to be representable is directly tied to the failure of the morphism  $X_1 \to X_2$ to be representable. We will be in good shape if we can find a sequence of simple stacky modifications that destackify $Y_1$ \textit{and} further descend to operations on $X_1$, which destackify it \textit{relative} to $X_2$. Indeed, we actually have an inertia stable groupoid presentation $(Y_1 \times_{X_1} Y_1 \double Y_1,c)$ of the stack $X_1$ by Proposition \ref{stablepullback} and Example \ref{stableexample}. Note that the presentation is defined over $Y_2$ even though $X_1$ is not. 

\begin{theorem}
\label{hatthm}
In the situation above, the following diagram exists:
\begin{equation}
\begin{tikzcd}[column sep=small]
& \hat{X_1} \arrow[dl] \arrow[dr] & \\
X_1 \arrow[dr] & & \hat{X}_{1,r. cs} \arrow[dl] \\
& X_2 
\end{tikzcd}
\end{equation}
The morphism $\hat{X_1} \to X_1$ is a composition of root stacks and smooth stacky blow-ups $X_n \to X_{n-1}$  along irreducible boundary divisors and such that if $D_{1,n}$ denotes the total transform of $D_1$ to $X_{n-1}$, then the center of the blow-up lies in $D_{1,n}$ and  intersects the strata of $D_{1,n}$ with normal crossings. The morphism $\hat{X_1} \to \hat{X}_{1,r. cs}$ is a composition of root stacks. The morphism are all isomorphisms over the fixed open set $U = X_1 - D_1$, and they induce a simple homotopy equivialence $$\Delta(X_1,D_1) \cong \Delta(\hat{X}_{1,r. cs},\hat{D_1}).$$
\end{theorem}

\begin{proof}
Since the groupoid $(Y_1 \times_{X_1} Y_1 \double Y_1,c)$ is inertia stable by remarks preceding the theorem, this is Corollary \ref{opdescent} combined with Lemma \ref{rootdescent} and Lemma \ref{centerdescent}.
\end{proof}

\subsection{\sc Step IV -  Reduction to the case of projective representable morphisms:}

To apply the functorial weak factorization theorem of Abramovich and Temkin  \cite{AT2}, we must reduce further to the case of representable projective morphisms between smooth Deligne-Mumford stacks. For this, we will need the following result.

\begin{theorem}
Let $f: X \to Y$ be a proper representable and birational morphism of smooth Deligne-Mumford stacks. If $f$ is an isomorphism on an open subset $U \subset X$, then there exists a smooth Deligne-Mumford stack $Z$ and a pair of \emph{projective} birational morphisms $g: Z \to X$ and $h: Z \to Y$ such that both $g$ and $h$ are isomorphisms on $g^{-1}(U)$ and $Z - U$ is a simple normal crossings divisor.
\label{eprop}
\end{theorem}

\begin{proof}

We construct the following diagram: 
\begin{equation}
\begin{tikzcd}
W^{(2)} \arrow [dr] \arrow[rr] \arrow[ddrr, bend right= 30]  \arrow[dddr, bend right = 30]  & & W^{(1)} \arrow[rr] \arrow{dd}[near start]{g''} & & W \arrow[dl,"g'"] \arrow[dddl, bend left =30,"h'"]  \\
& X\arrow[dd, crossing over] \arrow[dr, "q"] \arrow[rr] && X_{cs} \arrow[dd,crossing over] \\
& & Y \times_{Y_{cs}} X_{cs}  \arrow[ur] \arrow[dr] \arrow[dl]\\
& Y \arrow[rr] && Y_{cs} \\
\end{tikzcd}
\end{equation}

Note that although $Y \times_{Y_{cs}} X_{cs}$ is not isomorphic to $X$, there is a natural morphism $q: X \to Y \times_{Y_{cs}} X_{cs}$. Since the map $X \to Y$ is representable, it follows that $q$ is representable. By the Keel-Mori Theorem \ref{keelmori}, the arrows $X \to X_{cs}$ and $Y \to Y_{cs}$ are proper and quasi-finite. It follows that $Y \times_{Y_{cs}}  X_{cs} \to X_{cs}$ is proper, and from \cite[Tag: 0CPT]{stacks} we may conclude that $q$ is proper. Likewise, since both $X \to X_{cs}$ and  $Y \times_{Y_{cs}}  X_{cs} \to X_{cs}$ are quasi-finite, the morphism $q$ is quasi-finite. Since $q$ is quasi-finite, proper, and representable, it is in fact a finite morphism.

Let $U'$ be the image of the open subset $U \subset X$ in $X_{cs}$. The map $f_{cs}$ is an isomorphism on $U'$ since $U' \times_{X_{cs}} X$ is f isomorphic to $U_{cs}$   \cite[Theorem 1.1]{conrad2}. By applying Hironaka's version of Chow's lemma \cite[Lemma 1.3.1]{AKMW}\cite{MR0393556}, we may find a smooth variety $W$ and a pair of projective birational morphisms $g': W \to X_{cs}$ and $h': W \to Y_{cs}$ such that $g'$ and $h'$ are isomorphisms over $g'^{-1}(U_{cs})$. This data can be pulled back along the morphism $Y \to Y_{cs}$  to give a pair of projective representable morphisms  $g'':  W^{(1)} \to Y \times_{Y_{cs}}  X_{cs}$ and $h'': W^{(1)} \to Y$ isomorphic over $q^{-1}(U)$ where $W^{(1)}$ is defined to be $Y \times_{Y_{cs}} W$

Now, we set $W^{(2)} = W^{(1)} \times_{ Y \times_{Y_{cs}}  X} X$. The natural morphism $W^{(2)} \to X$ is projective and representable as it is the pullback of $g''$ and the natural morphism $W^{(2)} \to W^{(1)}$ is a \textit{projective} and representable as it is the pullback of $q$. These maps are all isomorphisms on the appropriate pullback of $U$. Thus, we have obtained a Deligne-Mumford stack $W^{(2)}$ equipped with birational \textit{projective} representable morphisms to both $X$ and $Y$. We may apply the stack desingularization theorem of Temkin \cite{Temkin} to obtain our desired stack $Z$.
\end{proof}

\subsection{\sc Step V -  Proof of the main theorem}

\begin{proof}
By Corollary \ref{MorphismReduction}, we may assume that there is a proper morphism $(X_1,D_1) \to (X_2,D_2)$ that is an isomorphism over $U$. By Theorem \ref{hatthm}, we may assume that $(X_1,D_1) \to (X_2,D_2)$ is proper and representable. By Theorem \ref{eprop}, we may find a pair $(Z,D)$ where $Z$ is a smooth Deligne-Mumford stack and $D$ is a snc divisor, and a diagram:
\begin{equation}
\begin{tikzcd}[column sep=small]
& (Z,D) \arrow[dl] \arrow[dr] & \\
(X_1,D_1) \arrow{rr} & & (X_2,D_2)
\end{tikzcd}
\end{equation}
where the downwards arrows are projective representable morphisms. Theorem \ref{main1} follows by applying \cite[Theorem 1.4.1]{AT2} and Remark \ref{keyremark} to the diagonal arrows. Theorem \ref{main2} follows from this and the results of Section 4.
\end{proof}

\appendix
\section{Cone and $\Delta$-complexes}
\label{appendix:a}

\subsection{Generalized boundary complexes} 
We recall here that a $\Delta$-complex is a functor $F$ from the category $\Delta_{inj}$ to the category of $Sets$. Here, $\Delta_{inj}$ is the category with one object $[p]$ for each integer $p$, and whose morphisms $[p] \to [q]$ are all injective order preserving maps $\{0,1,\ldots,p\}\to \{0,1,\ldots,q\}$. We shall say that the $\Delta$-complex $F$ is \textit{finite} if $F([n])$ is a finite set for each $[n]$.

\begin{adef}
Let $X$ and $Y$ be $\Delta$-complexes and set $X\cong_{simple}Y$ if $X$ and $Y$ are related by a collapse. Here, a collapse is consists of two faces $\sigma \in X$ and $\tau \in X$ such that $\tau \subset \sigma$ and such that $\tau$ is a maximal face of $X$ and no other maximal face contains $\sigma$. We say that $X$ collapses to $Y$ if $Y$ is isomorphic to the subcomplex $X' \subset X$ obtained by removing all faces $\gamma$ where $\tau \subseteq \gamma \subseteq \sigma$. We shall call the data $(\sigma,\tau)$ satsifying the above properties a \textit{classical collapse datum}. The equivalence relation generated by $\cong_{simple}$ is known as \textit{simple homotopy equivalence} \cite{MR0035437}.
\end{adef}

\begin{adef}\cite[Section 3.1]{CGP}
A \textit{generalized $\Delta$-complex} is a presheaf $F: I^{op} \to Set$ on the category $I$ where $Objects(I)=Objects(\Delta_{inj})$,  but the morphisms $[p]\to[q]$ correspond to all possible injective maps  $\{0,1,\ldots,p\}\to \{0,1,\ldots,q\}$. If the action of the symmetric group $S_{p+1}$ on $F([p],[p])$ is free for all $p$, then we say that $F$ is an \textit{unordered $\Delta$-complex} rather than a generalized $\Delta$-complex. In either case, we say that $F$ is finite if $F([n])$ is a finite set for each $[n]$.
\end{adef}

Geometrically, generalized $\Delta$-complexes are analagous to spaces built from quotients of the sphere $S^n / E$ where $E \subset S_n$ is a subgroup of the symmetric group $S_n$.  For the purposes of this paper, we must generalize the notion of simple homotopy equivalence to the category of generalized $\Delta$-complexes. Note, as in \cite[Section 3.1]{CGP}, that every $\Delta$-complex $X$ gives rise to an unordered $\Delta$-complex $X_{un}$ by forgetting the ordering. We also note that the definition of simple homotopy equivalence extends immediately to the category of unordered $\Delta$-complexes as it does not take into account the structure of the ordering.

\subsection{Generalized cone complexes}

In this section, we recall the equivialence between \textit{generalized boundary complexes} and a special class of \textit{generalized cone complexes}. We use this equivalence to translate existing barycentric subdivision algorithms (Lemma \ref{comlem}) into our setting.

\begin{adef}\cite[Section 2.6]{ACP}
A \textit{generalized cone complex} is a topological space $X$ together with a category $C$ and a functor $F:C \to Con$ such that $X$ arises as the colimit $colim(r\circ F)$. Here $Con$ is the category of Cones together with face morphisms and $r: Con \to Top$ is the forgetful functor.
\end{adef}

\begin{adef}\cite[Section 3.4]{CGP}
A generalized cone complex is \textit{smooth} if it may be presented as the colimit of standard orthants $\mathbb{R}_{\geq 0} ^m$ together with face morphisms.
\end{adef}

The connection between generalized cone complexes and generalized $\Delta$-complexes is elaborated on in  \cite[Section 3.4]{CGP}:

\begin{alem}
\label{equivcat}
\cite[Section 3.4]{CGP} There is an equivalence of categories between generalized $\Delta$-complexes with a unique vertiex and the category $SGC$ whose objects are smooth generalized cone complexes with a unique vertex and whose morphisms are face morphisms between such complexes.
\end{alem}

These ideas allow us to pass from the topological world of generalized $\Delta$-complexes to the more toric world of $SGC$. In particular, it allows us to use the toric and toroidal of constructions of \cite[Section 3.4]{AT2}

\begin{alem} \cite[Section 3.4.1]{AT2}
\label{comlem}
\begin{enumerate}
\item \cite[Section 3.4.1]{AT2} \cite[Section 2.5]{ACP} The barycentric subdivisions $B(\Delta)$ of a generalized cone complex $\Delta$ is a projective subdivision obtained by a sequence of simultaneous star subdivisions. If $\Delta$ is nonsingular then the star subdivisions are smooth and the generalized cone complex $B(\Delta)$ is in fact a cone complex.
\item \cite[Section 3.4.1]{AT2}  \cite[Lemma 8.7]{AMR}The barycentric subdivisions $B(\Delta)$ of a smooth cone complex $\Delta$ is a projective subdivision obtained by a sequence of simultaneous smooth star subdivisions. The smooth cone complex $B(\Delta)$ is in fact isomorphic to a fan.
\end{enumerate}
\end{alem}

In the case where $\Delta$ is the generalized cone complex of a toroidal scheme, then it is known (see \cite[Page 14]{AT2} for a detailed discussion) that the simultaneous star subdivision operations in \ref{comlem} above may be realized as toroidal morphisms of toroidal schemes. In view of the above theorem, the following definition becomes natural:

\begin{adef}
\label{sht}
The \textit{simple homotopy type} of a generalized $\Delta$-complex $X$ is the simple homotopy type of the composition $B(B(GC(X)))$ where $GC$ denotes the generalized cone associated with $X$ and $B$ is the barycentric subdivision operator.
\end{adef}

\begin{alem}
\label{redsnc}
Assume (X,E) is a smooth Deligne-Mumford stack together with a normal crossings divisor $E$, then there exists a proper birational morphism $(X',E') \to (X,E)$ that is an isomorphism over $X-E$ and such that $E'$ is a simple normal crossings divisor. Moreover, the simple homotopy type of the unordered $\Delta$-complex $\Delta(X',E')$ is the same as the simple homotopy type of the generalized $\Delta$-complex $\Delta(X,E)$
\end{alem}
\begin{proof}
Note that if $(X,E)$ admits an \'etale presentation $E_1 \double E_0$, then the generalized $\Delta$-complex $\Delta(X)$ is just the homotopy type of $$B(B(GC(Colim(\Delta(E_1) \double \Delta(E_0)))))$$ 
Using Lemma \ref{equivcat}, we may instead work with generalized cone complexes $\Delta^{cone}(X)$, $\Delta^{cone}(E_0)$, and $\Delta^{cone}(E_1)$. We also note that the two morphisms between $\Delta^{cone}(E_1)$ and $\Delta^{cone}(E_0)$ are surjective face morphisms, as is the quotient morphism $\Delta^{cone}(E_0) \to \Delta^{cone}(X)$. Barycentric subdivision is functorial with respect to surjective face morphisms and it follows that $B(\Delta^{cone}(X))$ is the colimit of $$colim(B(\Delta^{cone}(E_1)) \double B(\Delta^{cone}(E_0)))$$

The barycentric subdivision operation may be realized as a sequence of star subdivisions on the generalized cone complex, and these  star subdivisions may be explicitly realized as the induced morphisms from a sequence of blow-up operations \cite[Page 14]{AT2}: $$X'=X_n \to X_{n-1} \to \ldots \to X_0 = X$$

The aforementioned star subdivisions may be factored functorially into collapses with respect to surjective face morphisms\cite[Page 14]{AT2}.   Further, the above sequence of blow-up operations induces a corresponding sequence of morphisms on the level of presentation $$(E^{n}_1 \double E^{n}_0) \to (E^{n-1}_1 \double E^{n-1}_0) \to \cdots (E_1 \double E_0)$$  
Since the star subdivisions corresponding to each individual morphism in the above sequence may be functorially factored into collapses (with respect to the projections), we obtain a simple homotopy equivalence between $$Colim(\Delta^{cone}(E_1) \double \Delta^{cone}(E_0))$$  and $$\Delta(X',E') = Colim(\Delta^{cone}(E^{n}_1,D_1)) \double \Delta^{cone}(E^{n}_0,D_0)) $$
$$= Colim(B(\Delta^{cone}(E_1,D_1))) \double B(\Delta^{cone}(E^{n}_0,D_0))$$ 

By \ref{comlem}, two applications of barycentric subdivision will yield a smooth Deligne-Mumford stack $X''$ together with a smooth fan complex, and in particular a simple normal crossings divisor. By construction, it is clear that $\Delta(X'',E'')$ has the same simple homotopy type as $\Delta(X)$.
\end{proof}

\bibliographystyle{plain}             
\bibliography{stacksfactor}

\end{document}